\documentclass[11pt]{amsart}
\usepackage{latexsym}
\usepackage{amsmath,amsthm,amssymb, amscd,color}
\usepackage{mathrsfs}
\usepackage[all]{xy}
\usepackage[all]{xy}
\usepackage{tikz}
\usepackage{tikz-cd}
\usepackage{adjustbox}
\usetikzlibrary{snakes}
\usepackage[utf8]{inputenc}
\usepackage{enumerate}
\usepackage{breqn}
\usepackage{hyperref}
\usepackage{listings}
\usepackage{graphicx}
\usepackage{float}
\usepackage{tikz}
\usepackage{xcolor}
\usetikzlibrary{shapes}
\tikzset{
    partial ellipse/.style args={#1:#2:#3}{
        insert path={+ (#1:#3) arc (#1:#2:#3)}
    }
}

\lstdefinestyle{mystyle}{
    language=Python,
    basicstyle=\small\ttfamily,
    showspaces=false,
    showstringspaces=false,
    numbers=left,
    numberstyle=\tiny,
    numbersep=5pt
}

\usepackage{cleveref}

\usepackage{blkarray} 
\usepackage{framed}

\numberwithin{equation}{section}

\theoremstyle{plain}
\newtheorem{theorem}{Theorem}[section]

\numberwithin{equation}{section}
\newtheorem{proposition}[theorem]{Proposition}

\theoremstyle{definition}
\newtheorem{definition}[theorem]{Definition}

\theoremstyle{remark}

\newcommand{\RR}{\mathbb{R}}

\newcommand{\SN}{\mathbb{S}}

\begin{document}

\title[Some remarks on vector fields in $\RR^2$]{Some remarks on polynomial vector fields in $\RR^2$}

\author[S. Jana]{Supriyo Jana}
\address{Department of Mathematics, Indian Institute of Technology Bombay, India}
\email{supriyojanawb@gmail.com}

\author[S. Sarkar]{Soumen Sarkar}
\address{Department of Mathematics, Indian Institute of Technology Madras, India}
\email{soumen@iitm.ac.in}

\date{\today}
\thanks{}

\abstract 
In this paper, we study planar polynomial vector fields admitting invariant circles. We characterize polynomial vector fields having $m$ concentric invariant circles, and we describe those having three invariant circles whose centers are not collinear. We also give a necessary and sufficient condition for the unit circle to be the unique limit cycle of a quadratic vector field. We prove that a cubic or quartic Kolmogorov vector field having the unit circle invariant has no circular limit cycle, and that a quartic vector field with two concentric circular limit cycles has at most three circular limit cycles, this bound being sharp.
\endabstract

\maketitle 

\section{Introduction and statement of the main results}
The qualitative theory of differential equations studies the global behavior of solutions without solving the equations explicitly. Among the systems considered there, the polynomial ones have received particular attention. We consider the following system of differential equations
\begin{equation} \label{eq: I1}
 \frac{dx}{dt} = P(x, y), \frac{dy}{dt} = Q(x, y),  \end{equation}
called a polynomial differential system in $\mathbb{R}^2$ where $P, Q \in \mathbb{R}[x,y]$. The differential operator
\begin{equation}  \label{eq: I2}
 \mathcal{X} = P \frac{\partial}{\partial x} + Q   \frac{\partial}{\partial y}
\end{equation}
is called the vector field associated with the system \eqref{eq: I1}. The degree of the polynomial vector field in \eqref{eq: I2} is defined as the maximum of the degrees of  $P$ and $Q$. 

An isolated periodic orbit in the set of all periodic orbits of the vector field \eqref{eq: I2} is called a limit cycle. If a limit cycle is an oval of an algebraic curve, it is called an algebraic limit cycle. Hilbert’s 16th problem \cite{hilbert2000mathematical} asks for the maximum number of limit cycles of a planar polynomial vector field as a function of the degree of the vector field. For instance, a degree one vector field does not exhibit limit cycles. Dulac's theorem \cite{Per01} states that a planar polynomial vector field has only finitely many limit cycles. Quadratic vector fields having 4 limit cycles were constructed in \cite{chen1979relative, shi1980concrete}. Gaiko \cite{gaiko2008limit} proved that a quadratic vector field has at most 4 limit cycles. So, this is the sharp bound for quadratic vector fields. Although Hilbert's 16th problem is more than 100 years old, still it is unsolved even for degree three. However, the problem motivated significant advancements in the geometric theory of planar differential equations, bifurcation theory, foliations, and some topics in algebraic geometry. Notable contributions to the Hilbert's 16th problem can be found in  \cite{caubergh2008hilbert,han2012lower,llibre201016th,prohens2018new,zhang201116th}.

 The system \eqref{eq: I1} is called a polynomial Kolmogorov system if $P=x P'$, and $Q=yQ'$ for some $P', Q' \in \RR[x, y]$. The associated vector field is called a polynomial Kolmogorov vector field. Planar Kolmogorov systems were introduced by Kolmogorov \cite{kolmogorov1936sulla} to describe the interaction between two species occupying the same ecological habitat. These systems may exhibit quite complex dynamical behaviors, see \cite{hirsch1989systems} and the references therein (e.g., references 6, 10, 11, and 12).

 An invariant algebraic set for \eqref{eq: I2} is a subset $A \subset \mathbb{R}^2$ such that $A$ is the zero set of some polynomial $f \in \mathbb{R}[x,y]$ and $\mathcal{X}f = Kf$ for some $K \in \mathbb{R}[x,y]$. The polynomial $K$ is called the cofactor of $f=0$. Moreover, $\mathcal{X}$ is also called a polynomial vector field on $A$. 
 
\begin{definition}
Let $U$ be an open subset of $\RR^2$. A non-constant smooth map $H \colon U \to \mathbb{R}$ is called a first integral
of the vector field \eqref{eq: I2} on $U$ if $H$ is constant on all solution curves of the system \eqref{eq: I1} contained in $U$; i.e., $H(x(t),y(t)) =$ constant for all values of $t$ for which the solution $(x(t),y(t))$ is defined and contained in $U$.
\end{definition}
Note that $H$ is a first integral of the vector field \eqref{eq: I2} on $U$ if and only if $\mathcal{X} H=0$ on $U$. According to the Darboux theory of integrability \cite{darboux1878memoire, DuLlAr06}, if a system in $\RR^2$ has sufficiently many invariant algebraic sets, then it has a rational first integral.

This paper is organized in the following way. In Section \ref{sec:pol-vflds-with-inv-circ}, we recall the characterization of polynomial (Kolmogorov) vector fields on $\SN^1$. We explicitly describe degree $n$ vector fields with $m(\leq \lfloor \frac{n}{2} \rfloor)$ invariant concentric invariant circles, see \Cref{inv-concentric-form}.
We describe the polynomial planar vector fields having three invariant circles whose centers make a triangle; see Proposition \ref{vertices-triangle}. In the rest of this section, we exhibit a few more properties of planar vector fields having some invariant circles, see Propositions \ref{thm_not_equiv} and \ref{prop:disprove_15}.

In Section \ref{sec:cubic}, we prove that quartic vector fields with two concentric circular limit cycles can have at most 3 circular limit cycles \Cref{thm:two-con-quartic}. We also show that the bound is sharp. We discuss circular limit cycles of cubic and quartic Kolmogorov vector fields with $\SN^1$ as an invariant circle, see Theorems \ref{thm:kolm-3} and \ref{thm-qt-kolm}.


\section{Polynomial vector fields with invariant circles} \label{sec:pol-vflds-with-inv-circ}
In this section, we explore several interesting properties of planar vector fields having invariant circles. We begin with recalling two known results.  

\begin{proposition}[\cite{MR4883650}]\label{prop:s1-char}
Let $\SN^1$ be an invariant set of a planar polynomial differential system of degree $m$. Then the system satisfies the following.
\begin{equation*}
    \dot{x}= (x^2+y^2-1)f +A y,~\dot{y}= (x^2+y^2-1)g-A x
\end{equation*} for some $f,g,A\in \RR[x,y]$ with $\deg f, \deg g\leq (m-2)$ and $\deg A\leq (m-1)$.
\end{proposition}

 \begin{proposition}\cite[Theorem 1.1]{JaSo27}\label{cr:kolm3}
 Suppose that $\mathcal{X}=(P,Q)$ is a Kolmogorov vector field of degree $n$ in $\RR^2$. Then $\mathcal{X}$ is a vector field on $\mathbb{S}^1$ if and only if there exist $\xi, \zeta, C \in \RR[x,y]$ with degrees less than $(n-2)$ such that
\begin{equation}\label{cubic-kolm-form}
P=x((x^2+y^2-1) \xi + C y^2)\quad \mbox{and} \quad Q=y((x^2+y^2-1)\zeta - C x^2). 
\end{equation}
\end{proposition}

\begin{proposition}\label{prop:s1-limit-cycle}
Let $\SN^1=x^2+y^2-1=0$ be a periodic orbit of a quadratic planar field $\mathcal{X}=(P, Q)$. Then, that circle is a limit cycle of $\mathcal{X}$ if and only if $a \alpha + b \beta \neq 0$ where $\alpha, \beta$ are the constant terms of $P, Q$ respectively, and $a, -b$ are the coefficients of $xy$ in $P, Q$ respectively. Moreover, the circle is a hyperbolic limit cycle and the vector field has no other limit cycle.
\end{proposition}
\begin{proof}
 The vector field can be written as follows by Proposition \ref{prop:s1-char}. 
$$P=\alpha(x^2+y^2-1)+Ay,~Q=\beta(x^2+y^2-1)-Ax$$ where $\alpha,\beta\in \RR$ and $A$ is a polynomial of degree at most one, say $A=ax+by+c$. The circle $\SN^1$ contains no singular points as it is a periodic orbit. So, the line $A=0$ does not intersect $\SN^1$. Hence, by \cite[Theorem 4]{dolov2012limit}, $\SN^1$ is a limit cycle if and only if $\alpha (-A_x)+\beta (-A)_y=-(a\alpha +b \beta) \neq 0$. Note that $a, -b$ are the coefficients of $xy$ in $P, Q$ respectively. Moreover, $\SN^1$ is a hyperbolic limit cycle and the vector field has no other limit cycle.

\end{proof}

We note that Hilbert's 16th problem for algebraic limit cycles has been solved for concentric circles in \cite{Dol} and for non-concentric circles in \cite{llibre201616th, llibre201016th}.
\begin{proposition}\cite[Lemma 3]{Dol}\label{thm:hilbert-concentric}
A planar polynomial vector field of degree $n$ can have at most $\lfloor \frac{n}{2}\rfloor$ algebraic limit cycles which are concentric circles. Moreover, this bound can be reached. 
\end{proposition}

In the next result, we classify planar vector fields having several concentric invariant circles. 
\begin{theorem}\label{inv-concentric-form}
A degree $n$ vector field $\mathcal{X}=(P, Q)$ has $m (\leq \lfloor \frac{n}{2} \rfloor)$ many concentric invariant circles if and only if 
\begin{equation}\label{n/2-circle-form}
    P=\tilde{f} \prod_{i=1}^{m}(x^2+y^2-r_i^2)+Cy\quad \mbox{and}\quad Q=\tilde{g} \prod_{i=1}^{m}(x^2+y^2-r_i^2)-Cx
\end{equation}
 for some $\tilde{f},\tilde{g},C\in \RR[x,y]$ with $\deg \tilde{f}, \deg \tilde{g}\leq n-2m$ and $\deg C\leq n-1$. 
\end{theorem}

\begin{proof}
Let $(0,0)$ be the common center of the invariant concentric circles and assume that $x^2+y^2-1=0$ is an invariant circle. So, by Proposition \ref{prop:s1-char}, we have the following. 
$$P=(x^2+y^2-1)f+Ay, \quad Q=(x^2+y^2-1)g-Ax$$ where $f, g, A\in \RR[x,y]$. Suppose that $C_r:=\{x^2+y^2-r^2=0\}=0$ is an invariant circle where $r\neq 1$. Then, $x^2+y^2-r^2$ divides $Px+Qy=(x^2+y^2-1)(fx+gy)$. Consequently, $x^2+y^2-r^2$ divides $fx+gy$. Thus, $(f,g)$ is a vector field on $C_r=0$. By applying  \Cref{prop:s1-char} for $\gamma(x,y)=x^2+y^2-r^2$, we can express
    $$f=(x^2+y^2-r^2)f_1+A_1 y,~g=(x^2+y^2-r^2)g_1-A_1 x$$ for some $f_1,g_1,A_1\in \RR[x,y]$. Therefore,
    $$P=(x^2+y^2-1)(x^2+y^2-r^2)f_1+By,~Q=(x^2+y^2-1)(x^2+y^2-r^2)g_1-Bx$$ for some $B\in \RR[x,y]$. Inductively, we obtain
    $$P=\tilde{f}\prod_{i=1}^{m}(x^2+y^2-r_i^2)+Cy\quad \mbox{and}\quad Q=\tilde{g}\prod_{i=1}^{m}(x^2+y^2-r_i^2)-Cx$$ for some $\tilde{f},\tilde{g},C\in \RR[x,y]$ with $\deg \tilde{f}, \deg \tilde{g}\leq n-2m$ and $\deg C \leq n-1$.

Conversely, one can verify that the vector field $(P,Q)$ has $m \leq \lfloor \frac{n}{2} \rfloor$ invariant concentric circles if $P$ and $Q$ are given by \eqref{n/2-circle-form}. 
\end{proof}

\begin{proposition}\label{vertices-triangle}
    Suppose that $\mathcal{X}=(P,Q)$ is a polynomial vector field such that it has invariant circle $f_i=0$ with cofactor $K_i$ and center $(a_i,b_i)$ for $i=1,2,3$. Assume that centers of the circles form a triangle. Then, $P$ and $Q$ have the following form.
\begin{equation*}
\begin{split}
        P&=\frac{1}{2\ell}((b_3-b_2)K_1f_1+(b_1-b_3)K_2f_2+(b_2-b_1)K_3f_3)\\
        Q&=-\frac{1}{2\ell}((a_3-a_2)K_1f_1+(a_1-a_3)K_2f_2+(a_2-a_1)K_3f_3)
\end{split}
\end{equation*}
    where $\ell=(a_1-a_2)(b_1-b_3)-(b_1-b_2)(a_1-a_3)$.
\end{proposition}
\begin{proof}
Let $f_i=(x-a_i)^2+(y-b_i)^2 -r_i^2$ for $i=1,2,3$. By the hypothesis, for $i=1,2,3,$ we have the relation $$2(x-a_i)P+2(y-b_i)Q=K_if_i,$$ which implies that
$$Px+Qy=a_iP+b_iQ+\frac{1}{2}K_if_i.$$
Hence, we obtain the following system of equations.
\begin{equation*}
    \begin{cases}
    \begin{split}
        a_1P+b_1Q+\frac{1}{2}K_1f_1&=a_2P+b_2Q+\frac{1}{2}K_2f_2\\
        a_1P+b_1Q+\frac{1}{2}K_1f_1&=a_3P+b_3Q+\frac{1}{2}K_3f_3.
    \end{split}
    \end{cases}
\end{equation*}
Thus,
$$\begin{pmatrix}
    a_1-a_2 & b_1-b_2\\
    a_1-a_3 & b_1-b_3
\end{pmatrix} \begin{pmatrix}
    P\\
    Q
\end{pmatrix}=\frac{1}{2}\begin{pmatrix}
    K_2f_2-K_1f_1\\
    K_3f_3-K_1f_1
\end{pmatrix}.$$ Let $\ell:=(a_1-a_2)(b_1-b_3)-(b_1-b_2)(a_1-a_3)$. Note that $\ell \neq 0$ because the centers of the circles form a triangle. Thus, $P$ and $Q$ have unique solution as follows.
$$\begin{pmatrix}
    P\\
    Q
\end{pmatrix}=\frac{1}{2\ell} \begin{pmatrix}
     b_1-b_3 & b_2-b_1\\
    a_3-a_1 & a_1-a_2
\end{pmatrix} \begin{pmatrix}
    K_2f_2-K_1f_1\\
    K_3f_3-K_1f_1
\end{pmatrix}.$$
This completes the proof.
\end{proof}

Let $f, g \in \RR[x,y]$. Then,  the determinant of Jacobian matrix of $f$ and $g$ is $\{f,g\}=\frac{\partial f}{\partial x} \frac{\partial g}{\partial y} - \frac{\partial f}{\partial y} \frac{\partial g}{\partial x}$. 

\begin{proposition}\label{s+1-circles-counter}
Let $f_1:= x^2+y^2-1, f_2:= (x-a)^2 + y^2+a^2-1$, $ f_j:= x^2+y^2-r^2_j$, for $3\le j \le s$, and $f_{s+1}:= (x-\frac{a}{2})^2 + y^2 +\frac{3a^2}{4}-1$, where $a\in (0,1)$. 
Then,$\{f_i, f_j\}$ is non-zero for $i, j \in \{1, \ldots, s+1\}$ with $i \neq j$. 
Moreover,  $f_1=0, \ldots, f_{s+1}=0$ are invariant circles for the following vector field.
\begin{equation}\label{equivalent-counter1}
    \begin{split}
        P&= (f_1 - f_2)f_3f_4\cdots f_s y\\
        Q &= -(f_1-f_2)f_3f_4\cdots f_s x + af_1f_3f_4\cdots f_s 
    \end{split}
\end{equation}
\end{proposition}
\begin{proof}
The first part follows from the fact that $\{f_i, f_j\}= 4ay\neq 0$. 

Now, $2(x-a)P+2y Q= 2a yf_2 \cdots f_s$, and $2xP+2y Q= 2a yf_1 f_3 \cdots f_s$. Also,
\begin{equation*}
    \begin{split}\label{equivalent-counter}
 2(x-\frac{a}{2})P + 2yQ &= ay(f_1 + f_2)f_3f_4\cdots f_s= 2ayf_{s+1}f_3f_4\cdots f_s 
    \end{split}
\end{equation*}
So, $f_j=0$ is an invariant circle for $(P, Q)$ where $j=1, \ldots, s+1$.
\end{proof}

We note that \cite[Theorem 1]{llibre201616th} states that $f_1=0, \ldots, f_{s+1}=0$ are invariant circles for the vector field 
\begin{equation}\label{eq_4inv_circ}
\begin{split}
P(x,y)&=-\lambda_{s+2}\prod_{m=1}^{s+1}f_m+\sum_{j=1}^{s+1}\lambda_j \Big(\prod_{m=1,m\neq j}^{s+1}f_m\Big)\{f_j,x\}\\
Q(x,y)&=\lambda_{s+3}\prod_{m=1}^{s+1}f_m+\sum_{j=1}^{s+1}\lambda_j \Big(\prod_{m=1,m\neq j}^{s+1}f_m\Big)\{f_j,y\},
\end{split}
\end{equation}
where $\lambda_1,\ldots, \lambda_{s+3}$ are polynomials in $\RR[x,y]$.

\begin{proposition}\label{thm_not_equiv}
There is no $\lambda_1, \ldots, \lambda_{s+3} \in \RR[x, y]$ such that the vector field in \eqref{equivalent-counter1} can be written as in \eqref{eq_4inv_circ}. 
\end{proposition}
\begin{proof}
Suppose that
 $\mathcal{X}=(P, Q)$ of \ref{equivalent-counter1} can be written in the form \eqref{eq_4inv_circ}. This leads to the following system of two linear equations in $\lambda_1$ and $\lambda_2$.
\begin{equation*}
    \begin{cases}
    \begin{split}
         -\frac{2y}{ f_1}\lambda_1 -  \frac{2y}{ f_2}\lambda_2 &= \lambda_{s+2} + 2y \sum_{j=3}^{s+1} \frac{\lambda_j}{f_j} + \frac{P}{\prod_{i =1}^{s+1}f_i} \\
        \frac{2x}{f_1} \lambda_1+ \frac{2(x-a)}{f_2} \lambda_2 &=-\lambda_{s+3} -2x \sum_{j=3}^s \frac{\lambda_j}{f_j} -2(x-\frac{ a}{2}) \frac{\lambda_{s+1}}{f_{s+1}} + \frac{Q}{\prod_{i=1}^{s+1}f_i} 
    \end{split}
    \end{cases}
\end{equation*}
Hence, we obtain
\begin{equation*}
    \begin{split}
        \lambda_2 &=\frac{f_1f_2}{4ay} \begin{vmatrix}
        -\frac{2y}{f_1} &
\lambda_{s+2} + 2y \sum_{j=3}^{s+1} \frac{\lambda_j}{f_j} + \frac{P}{\prod_{i =1}^{s+1}f_i} \\
\frac{2x}{f_1} &  -\lambda_{s+3} -2x \sum_{j=3}^s \frac{\lambda_j}{f_j} -2(x-\frac{ a}{2}) \frac{\lambda_{s+1}}{f_{s+1}} + \frac{Q}{\prod_{i=1}^{s+1}f_i}
\end{vmatrix}\\
&=\frac{f_2}{2ay} ((\lambda_{s+3} y-\lambda_{s+2} x) - ay \frac{\lambda_{s+1}}{f_{s+1}} - \frac{Px+Qy}{\prod_{i=1}^{s+1} f_i})\\
&=\frac{f_2}{2ay} ((\lambda_{s+3} y-\lambda_{s+2} x) - ay \frac{\lambda_{s+1}}{f_{s+1}} - \frac{ay}{f_2f_{s+1}})\\
&=\frac{1}{2ay}(f_2(\lambda_{s+3} y-\lambda_{s+2} x) - \frac{ay(f_2\lambda_{s+1}+ 1)}{f_{s+1}})
    \end{split}
\end{equation*}
So, $\lambda_2 \in \RR[x,y]$ if and only if $\frac{ay(f_2\lambda_{s+1}  +1)}{ f_{s+1}} \in \RR[x,y]$. Suppose that $q(x,y):=\frac{ay(f_2\lambda_{s+1}  +1)}{ f_{s+1}}\in \RR[x,y]$. Then, $q(a,y)\in \RR[y]$, which implies that $\frac{ay}{y^2+a^2-1}\in \RR[y]$. This arises a contradiction. Therefore, $\lambda_2\notin \RR[x,y]$. Hence, the proof is complete.
\end{proof}

\begin{proposition}\label{prop:disprove_15}
Let $f_1:= x^2+y^2-1, f_2:= (x-a)^2 + y^2+a^2-1$, $ f_3:= (x+\frac{a}{2})^2+y^2-(1+\frac{5a^2}{4})$, and $f_4:= (x-\frac{a}{2})^2 + y^2 +\frac{3a^2}{4}-1$, where $a\in (0,\frac{2}{\sqrt{3}})$. Then, the degree 4 vector field $\mathcal{X}=(P,Q)$ with 
\begin{equation}\label{equivalent-counter2}
    \begin{split}
        P&= (f_1 - f_2)f_3 y\\
        Q &= -(f_1-f_2)f_3 x + af_1f_3 
    \end{split}
\end{equation}
has 4 invariant circles $f_1=0,\ldots,f_4=0$. Moreover,
there exist no $\lambda_1,\ldots,\lambda_4\in \RR$ such that $\mathcal{X}=\prod\limits_{i=1}^4 f_i\sum\limits_{n=1}^4 \lambda_n \{\log |f_n|,*\}$.
\end{proposition}
\begin{proof}
Using similar arguments as in the proof of \Cref{s+1-circles-counter}, one can prove that $f_1=0,\ldots,f_4=0$ are invariant circles for the vector field $\mathcal{X}=(P,Q)$ where $P$ and $Q$ are given by \eqref{equivalent-counter2}.

Suppose that $\mathcal{X}=(P,Q)$ of \eqref{equivalent-counter2} can be expressed as $\mathcal{X}=\prod\limits_{i=1}^4 f_i\sum\limits_{n=1}^4 \lambda_n \{\log |f_n|,*\}$. This leads to a system of two linear equations in $\lambda_1,\lambda_2$.
\begin{equation*}
    \begin{cases}
    \begin{split}
         -\frac{2y}{ f_1}\lambda_1 -  \frac{ 2y}{ f_2}\lambda_2 &= 2y \Big(\frac{\lambda_3}{f_3}+\frac{\lambda_4}{f_4}\Big) + \frac{P}{\prod_{i =1}^{4}f_i} \\
        \frac{2x}{f_1} \lambda_1+ \frac{2(x-a)}{f_2} \lambda_2 &=-2(x+\frac{a}{2})\frac{\lambda_3}{f_3} -2(x-\frac{ a}{2}) \frac{\lambda_4}{f_4} + \frac{Q}{\prod_{i=1}^{4}f_i}
    \end{split}
    \end{cases}
\end{equation*}
Hence, we obtain
\begin{equation*}
    \begin{split}
        \lambda_2 &=\frac{f_1f_2}{4ay} \begin{vmatrix}
        -\frac{2y}{f_1} & 2y \Big(\frac{\lambda_3}{f_3}+\frac{\lambda_4}{f_4}\Big) + \frac{P}{\prod_{i =1}^{4}f_i} \\
\frac{2x}{f_1} &  -2(x+\frac{a}{2})\frac{\lambda_3}{f_3} -2(x-\frac{ a}{2}) \frac{\lambda_4}{f_4} + \frac{Q}{\prod_{i=1}^{4}f_i}
\end{vmatrix}\\
&=-\frac{f_2}{ay} ( \frac{ay}{2} (-\frac{\lambda_3}{f_3}+\frac{\lambda_4}{f_4}) + \frac{Px+Qy}{\prod_{i=1}^4 f_i})\\
&=-\frac{f_2}{ay} ( \frac{ay}{2} (-\frac{\lambda_3}{f_3}+\frac{\lambda_4}{f_4}) + \frac{ay}{f_2f_4})\\
&=-f_2(-\frac{\lambda_3}{f_3}+\frac{\lambda_4}{f_4}) - \frac{1}{f_4}.
    \end{split}
\end{equation*}
For $x=a+\sqrt{1-a^2},y=0$ and $a-\sqrt{1-a^2},y=0$, we obtain $\lambda_2$ as $-\frac{1}{(\frac{a}{2}+\sqrt{1-a^2})^2 +\frac{3a^2}{4}-1}$ and $-\frac{1}{(\frac{a}{2}-\sqrt{1-a^2})^2 +\frac{3a^2}{4}-1}$, respectively. However, this is impossible since $\lambda_2$ is a constant. Hence, the proof is complete.
\end{proof}

\section{Circular limit cycles of cubic and quartic vector fields} \label{sec:cubic}
In this section, We study circular limit cycles of cubic and quartic planar vector fields. We recall that a quartic polynomial vector field can have at most 2 concentric limit cycles; see \Cref{thm:hilbert-concentric}.

\begin{theorem}\label{thm:two-con-quartic}
If a planar quartic vector field has two concentric circular limit cycles, then the vector field has at most three circular limit cycles. Moreover, the bound is sharp.
\end{theorem}

\begin{proof}
Let $\mathcal{X}=(P, Q)$ be a planar quartic vector field having the circular limit cycles $f_1=0, \ldots, f_4=0$, where $f_1:=x^2+y^2-1$ and $f_2:=x^2+y^2-r_2^2$ with $r_2\neq 1$. By \Cref{inv-concentric-form}, we have
\begin{equation*}
    P = \alpha f_1f_2  + A y\quad \mbox{and}\quad Q= \beta f_1f_2  - A x.
\end{equation*}
for some $\alpha,\beta\in \RR$ and $A\in \RR[x,y]$ with $\deg A\leq 3$. If necessary, by a rotation, we may assume that $f_3:=(x-a)^2+y^2-r_3^2$. By \Cref{thm:hilbert-concentric}, $a\neq 0$. Thus, there exists $K_3\in \RR[x,y]$ such that
\begin{equation}\label{quartic-g3-inv}
    (x-a) (\alpha f_1f_2  + A y) + y( \beta f_1f_2  - A x ) = K_3f_3.
\end{equation}
On $y = 0$, \eqref{quartic-g3-inv} gives 
\begin{equation}\label{quartic-g3-eq2}
   \alpha (x-a) ( x^2 - 1) ( x^2 - r^2_2)=K_3(x,0)((x-a)^2-r_3^2).
\end{equation}
If $x\pm 1$ divides $((x-a)^2-r_3^2)$, then $f_1=0$ intersects $f_3=0$ at $(\mp 1,0)$. Furthermore, if $x\pm r_2$ divides $((x-a)^2-r_3^2)$, then $f_2=0$ intersects $f_3=0$ at $(\mp r_2,0)$. Therefore, $(x^2 - 1) ( x^2 - r^2_2)$ divides $K_3(x,0)$. Therefore, $K_3(x, 0) = 0$, since $\deg K_3\leq 3$. Thus, $ K_3 = \widehat{K_3}y$ for some $\widehat{K_3} \in \RR[x, y]$ with $\deg \widehat{K_3} \leq 2$. Consequently, \eqref{quartic-g3-eq2} gives $\alpha=0$. Hence, \eqref{quartic-g3-inv} simplifies to
\begin{equation*}
        (x-a)A + (\beta f_1f_2  - A x) = \widehat{K_3}f_3,
\end{equation*}
which leads to
$$A = \frac {1}{a}(\beta f_1f_2  - \widehat{K_3}f_3).$$
Hence, we have
\begin{equation}\label{new-form-quartic-two-concentric}
    P = \frac {1}{a} ( \beta f_1f_2 - \widehat{K_3}f_3) y\quad \mbox{and}\quad Q = -\frac{x-a}{a}\beta f_1f_2 + \frac {x}{a}\widehat{K_3}f_3,
\end{equation}
where $\beta \neq 0$, since $f_3$ is a limit cycle. 
Now, suppose that $f_4=(x-b)^2+(y-c)^2-r_4^2$. Then, there exists $K_4\in \RR[x,y]$  with $\deg K_4 \leq 3$ such that
\begin{equation*}
        \frac {(x-b)}{a} ( \beta f_1f_2  - \widehat{K_3}f_3)y -\frac{(x-a)(y-c)}{a}\beta f_1f_2+\frac {x(y-c)}{a}\widehat{K_3}f_3 = K_4f_4,
\end{equation*}
which simplifies to
\begin{equation}\label{quartic-g4-inv}
    \frac{\beta}{a}f_1f_2(cx-by+a(y-c))-\frac{1}{a}(cx-by)\widehat{K_3}f_3=K_4f_4.
\end{equation}

If $c \neq 0$, then substituting $x = \frac{by}{c}$ in \eqref{quartic-g4-inv} yields 
$$\frac{\beta}{a}((\frac{b^2}{c^2} + 1)y^2 - 1) ((\frac{b^2}{c^2} + 1)y^2 - r^2_2) a(y-c) =K_4(\frac{by}{c},y) ((\frac{b^2}{c^2}+1)(y-c)^2 - r^2_4).$$

If $(y+ \frac{c}{\sqrt{b^2+c^2}})$ or $(y- \frac{c}{\sqrt{b^2+c^2}})$ divides $((\frac{b^2}{c^2}+1)(y-c)^2 - r^2_4)$ then $f_1=0$ and $f_4=0$ intersect at either $(\frac{b}{\sqrt{b^2+c^2}},\frac{c}{\sqrt{b^2+c^2}})$ or $(-\frac{b}{\sqrt{b^2+c^2}},-\frac{c}{\sqrt{b^2+c^2}})$. Furthermore, if $(y+ \frac{r_2c}{\sqrt{b^2+c^2}})$ or $(y- \frac{r_2c}{\sqrt{b^2+c^2}})$ divides $((\frac{b^2}{c^2}+1)(y-c)^2 - r^2_4)$ then $f_2=0$ and $f_4=0$ intersect either at $(\frac{r_2b}{\sqrt{b^2+c^2}},\frac{r_2c}{\sqrt{b^2+c^2}})$ or $(-\frac{r_2b}{\sqrt{b^2+c^2}},-\frac{r_2c}{\sqrt{b^2+c^2}})$.
Therefore, $((\frac{b^2}{c^2} + 1)y^2 - 1) ((\frac{b^2}{c^2} + 1)y^2 - r^2_2)$ divides $K_4(\frac{by}{c},y)$. Thus,  $K_4(\frac{by}{c}, y)=0$, since $\deg K_4\leq 3$. Consequently \eqref{quartic-g4-inv} gives $\beta=0$. Hence, $x^2+y^2$ is a first integral of $\mathcal{X}$. This contradicts the existence of limit cycles $f_1=0,\ldots, f_4=0$. So, $c$ must be zero. Moreover, we assume that $\beta\neq 0$. By \Cref{thm:hilbert-concentric}, we also have $b\neq 0$. 
Therefore, \eqref{quartic-g4-inv} becomes 
\begin{equation}\label{quartic-c-0}
    \frac{\beta}{a} f_1f_2 y(a-b) + \frac{by}{a} \widehat K_3 f_3 = K_4 f_4.
\end{equation}
Observe that $y$ is a factor of $K_4$, say $K_4=y\widehat{K_4}$  with $\deg \widehat{K_4} \leq 2$. Thus, \eqref{quartic-c-0} can be rewritten as
\begin{equation}\label{rewritten}
    \frac{\beta (a-b)}{a} f_1f_2 + \frac{b}{a} \widehat K_3 f_3 = \widehat{K_4} f_4.
\end{equation}
Suppose that $2K_1$ and $2K_2$ are the cofactors of $f_1=0$ and $f_2=0$, respectively. Hence, $K_1=\beta y f_2$ and $K_2=\beta yf_1$. 
Note that $\deg(P)\leq 4$. Hence, \eqref{new-form-quartic-two-concentric} gives $$\beta (x^2+y^2)^2 = \widehat K_3^{(2)} (x^2+y^2) \quad \implies \widehat{K_3}^{(2)} = \beta (x^2+y^2).$$
Now, comparing the highest degree terms in \eqref{rewritten}, we get $$ \frac{\beta (a-b)}{a} (x^2+y^2)^2 + \frac{\beta b}{a}  (x^2+y^2)^2 = \widehat{K_4}^{(2)} (x^2+y^2).$$ So, $\widehat{K_4}^{(2)} = \beta(x^2+y^2)$ and hence $\widehat K_3^{(2)} = \widehat K_4^{(2)}$. Comparing the degree 3 terms in \eqref{rewritten}, we obtain $$\frac{b}{a} \widehat {K_3}^{(2)} (-2ax) + \frac{b}{a} \widehat K_3^{(1)} (x^2+y^2) =\widehat K_4^{(2)} (-2bx)+\widehat K_4^{(1)} (x^2+y^2).$$ This gives $b \widehat K_3^{(1)} = a\widehat K_4^{(1)}$. Also, comparing the degree one terms in \eqref{rewritten}, we obtain $$\frac{b}{a} \widehat K_3^{(1)} (a^2 - r^2_3) + \frac{b}{a} \widehat K_3^{(0)} (-2ax) =\widehat K_4^{(1)} (b^2-r^2_4) + \widehat K_4^{(0)} (-2bx).$$ Hence, 
\begin{equation}\label{k_3,1}
    \widehat K_3^{(1)} (a^2 - b^2-r_3^2 +r_4^2) = 2ax (\widehat K_3^{(0)}-\widehat K_4^{(0)}).
\end{equation}

\noindent\textbf{Case-1:} Suppose that $a^2 - b^2-r_3^2 +r_4^2=0$. Then, $\widehat K_3^{(0)}=\widehat K_4^{(0)}$. Comparing the degree 2 terms in \eqref{rewritten}, we get
\begin{dmath}\label{case-1}
    \frac{\beta(a-b)}{a}(x^2+y^2)(-r_2^2-1)+\frac{b}{a}\widehat{K_3}^{(2)}(a^2-r_3^2)+\frac{b}{a}\widehat{K_3}^{(1)}(-2ax)+\frac{b}{a}\widehat{K_3}^{(0)}(x^2+y^2)=\widehat{K_4}^{(2)}(b^2-r_4^2)+\widehat{K_4}^{(1)}(-2bx)+\widehat{K_4}^{(0)}(x^2+y^2).
\end{dmath}
Note that $\widehat{K_3}^{(2)}=\beta (x^2+y^2)$. Rewriting \eqref{case-1}, we have
$$\beta (x^2+y^2) (-\frac{a-b}{a}(r_2^2+1)+\frac{b}{a}(a^2-r_3^2)-(b^2-r_4^2)+\frac{1}{\beta}(\frac{b}{a}-1)K_3^{(0)}
)=2bx(\widehat{K_3}^{(1)}-\widehat{K_4}^{(1)}).$$
This implies that $x$ divides $x^2+y^2$, which is not possible. Hence, $\widehat{K_3}^{(1)}-\widehat{K_4}^{(1)}=0$. Therefore,  $\widehat{K_3}=\widehat{K_4}$ since $\widehat{K_3}^{(2)}=\widehat{K_4}^{(2)}$. Therefore, $f_3f_4^{-1}$ is a first integral of $\mathcal{X}$.

\noindent\textbf{Case-2:} Suppose that $a^2 - b^2-r_3^2 +r_4^2\neq 0$. Then, $\widehat K_3^{(1)}= \frac{2a}{a^2 - b^2-r_3^2 +r_4^2} (\widehat K_3^{(0)}-\widehat K_4^{(0)})x$. Next, we find if there exist $\lambda_1,\ldots,\lambda_4\in \RR$, not all zero such that $\sum_{i=1}^4 \lambda_iK_i=0$. If such $\lambda_i$'s exist then
\begin{equation*}
        \lambda_1 \beta f_2 +\lambda_2 \beta f_1 +\lambda_3 \widehat{K_3}+\lambda_4 \widehat{K_4}=0,
\end{equation*}
which is satisfied if and only if
\begin{equation}\label{existence-lambda_i}
    \beta (\sum_{i=1}^4 \lambda_i)(x^2+y^2) + (\lambda_3+ \frac{b}{a}\lambda_4 )\frac{2a(\widehat K_3^{(0)}-\widehat K_4^{(0)})}{a^2 - b^2-r_3^2 +r_4^2} x +(- \beta r_2^2\lambda_1- \beta \lambda_2 +  \widehat K_3^{(0)} \lambda_3 +  \widehat K_4^{(0)} \lambda_4) = 0. 
\end{equation}
We take $\lambda_3=-\frac{b}{a}$ and  $\lambda_4=1$. Consider the following system of equations in $\lambda_1$ and $\lambda_2$.
\begin{equation*}
    \begin{cases}
    \begin{split}
        &\lambda_1+\cdots+\lambda_4=0, \\
        - \beta r_2^2\lambda_1- \beta \lambda_2 +  \widehat K_3^{(0)} &\lambda_3 +  \widehat K_4^{(0)} \lambda_4=0 .
        \end{split}
    \end{cases}
\end{equation*}
This system simplifies to 
\begin{equation*}
    \begin{cases}
    \begin{split}
        \lambda_1+\lambda_2=\frac{b}{a}-1,\\
          \beta r_2^2\lambda_1+ \beta \lambda_2 =\widehat K_4^{(0)} -\frac{b}{a} \widehat K_3^{(0)}.
  \end{split}
    \end{cases}
\end{equation*}
Observe that $\begin{vmatrix}
    1&1\\
    \beta r_2^2 & \beta
\end{vmatrix}=\beta (1-r_2^2)\neq 0.$ Hence, the above system of linear equations in $\lambda_1, \lambda_2$ has a unique solution, say $\lambda_1=p, \lambda_2=q$. So, $pK_1+qK_2-\frac{b}{a}K_3+K_4=0$. By Darboux Integrability Theory, $f_1^p f_2^q f_3^{(-\frac{b}{a})} f_4$ is a Darboux first integral of $\mathcal{X}$. Thus, $f_1=0,\ldots,f_4=0$ cannot be limit cycles. This completes the proof. 

The second part follows from \cite[Proposition 21]{llibre201616th}.
\end{proof}

\begin{theorem}\label{thm:kolm-3}
    If $x^2+y^2-1=0$ is an invariant circle of a planar cubic Kolmogorov vector field, then the vector field has no circular limit cycle.
\end{theorem}
\begin{proof}
Suppose that $\mathcal{X}=(P,Q)$ is a cubic Kolmogorov vector field having the circle  $x^2+y^2-1=0$ invariant. By \Cref{cr:kolm3},
    \begin{equation*}
P(x,y) =  x( \alpha(x^2+y^2-1) + \gamma y^2)~\mbox{and}~Q(x,y) = y(\beta (x^2+y^2-1) - \gamma x^2)
    \end{equation*}
where $\alpha, \beta, \gamma \in \RR$. Note that $x^2+y^2-1=0$ is not a limit cycle since $(1,0)$ is a singular point of $\mathcal{X}$. Suppose that $\mathcal{X}$ has a circular limit cycle $f_2=0$ where $f_2:=(x-a)^2+(y-b)^2-r^2$ with $a,b\neq 0$ since $f_2=0$ cannot intersect $x$ or $y$-axis.. Then, there exists $K_2\in \RR[x,y]$ such that 
\begin{equation}\label{kolmo-cubic}
    \begin{split}
      &(x-a)P+(y-b)Q=K_2f_2\\
    \implies&    (x-a)\alpha x (x^2+y^2-1) + (x-a)\gamma xy^2+ (y-b) \beta y (x^2+y^2-1) - (y-b) \gamma x^2y= K_2f_2\\
\implies& (\alpha x^2 - a\alpha x + \beta y^2 - b\beta y)(x^2+y^2-1) + \gamma xy (bx - ay)=K_2f_2
    \end{split}
\end{equation}
On $x=0$, the above equation gives
$$\beta y (y - b ) (y^2 -1) = K_2(0,y) ((y-b)^2 + a^2 - r^2).$$
If $y\pm 1$ divides $((y-b)^2 + a^2 - r^2)$, then $x^2+y^2-1=0$ and $f_2=0$ intersect at $(0,\mp 1)$. In this case, $f_2=0$ can not be a limit cycle. Hence, $y^2-1$ divides $K_2(0,y)$. If $y(y-b)$ divides $((y-b)^2 + a^2 - r^2)$ then $b=0$. So, $f_2=0$ will intersect $x$-axis. As $y=0$ is invariant for the vector $\mathcal{X}$ and $f_2=0$ is a limit cycle, this is not possible. So, at least one of $y$ and $(y-b)$ divides $K_2(0,y)$. Observe that $\deg K_2\leq 2$. Hence, $K_2(0,y)=0$ and consequently, $\beta=0$.

Also, on $y=0$, \eqref{kolmo-cubic} becomes
$$\alpha x(x-a)(x^2-1)=K_2(x,0)((x-a)^2+b^2-r^2).$$ 
Using similar arguments as in the above, we can prove that $\alpha=0$. Then, $x^2+y^2$ is a first integral of $\mathcal{X}$ since $\alpha,\beta=0$. This contradicts that $f_2=0$ is a limit cycle. Therefore, $\mathcal{X}$ cannot have any circular limit cycle.
\end{proof}

\begin{theorem}\label{thm-qt-kolm}
If $x^2+y^2-1=0$ is an invariant circle of a planar quartic Kolmogorov vector field, then the vector field has no circular limit cycle. 
\end{theorem}

\begin{proof}
Suppose that $\mathcal{X}=(P,Q)$ is a quartic Kolmogorov vector field having the circle $f_1:=x^2+y^2-1=0$ invariant. By \Cref{cr:kolm3},
\begin{equation*}
P = x((x^2+y^2-1)f + Ay^2)~\mbox{and}~Q= y((x^2+y^2-1)g - Ax^2)
\end{equation*}
where $f, g, A$ are polynomials of degree at most one. Suppose that $\mathcal{X}$ has a circular limit cycle $f_2=0$ where $f_2:=(x-a)^2+(y-b)^2-r^2$. Note that $a,b\neq 0$ since $f_2=0$ cannot intersect $x$ or $y$-axis. Then there exists $K_2\in \RR[x,y]$ with $\deg K_2 \leq 3$ such that
$$(x-a) P+ (y-b)Q = K_2((x-a)^2 + (y-b)^2 - r^2).$$ Hence, 
\begin{equation}\label{kolm-deg-4}
    (x^2+y^2-1)(x(x-a)f + y(y-b)g) + Axy(bx-ay)=K_2((x-a)^2 + (y-b)^2 - r^2)
\end{equation}
On $x= 0$, \eqref{kolm-deg-4} gives $$(y^2-1)y (y-b) g (0,y) = K_2(0,y)(((y-b)^2 +a^2- r^2).$$ If $(y\pm 1)$ divides $((a^2 + (y-b)^2 - r^2)$, then $f_1=0$ and $f_2=0$ intersect at $(0,\mp 1)$. So, $y^2 - 1$ divides $K_2(0,y)$. If $y-b$ divides $((a^2 + (y-b)^2 - r^2)$, then at $(0,b)$, the x-axis and $f_2=0$ intersect. Also, if $y$ divides $((a^2 + (y-b)^2 - r^2)$, then at $(0,0)$, the $x$-axis and $f_2=0$ intersect. Note that $y=0$ is invariant for the vector field $\mathcal{X}$. Hence, both cases contradict the fact that $f_2=0$ is a limit cycle. 
So, $ y(y-b) $ also divides $ K_2(0,y)$. Thus, $ K_2(0,y) = 0,$ since $\deg K_2 \le 3 $. So, $x$ divides $K_2$. 

On $y= 0$, \eqref{kolm-deg-4} gives $(x^2-1)x (x-a) f (x,0) = K_2(x,0)((x-a)^2 +b^2- r^2)$. By similar arguments as above, $ y$ divides $K_2$. Therefore, $xy$ divides  $K_2$. Hence, from \eqref{kolm-deg-4}, $xy $ divides  $ x(x-a) f + y(y-b)g $. This gives $f=\alpha y$ and $g=\beta x$ for some $\alpha,\beta \in \RR$. On $x=\frac{ay}{b}$, \eqref{kolm-deg-4} becomes
$$((\frac{a^2}{b^2}+1)y^2-1)y(y-b)(\frac{a^2}{b^2}f(\frac{ay}{b},y)+g(\frac{ay}{b},y))=K_2(\frac{ay}{b},y)((\frac{a^2}{b^2}+1)(y-b)^2 - r^2).$$ Observe that $y-b$ divides $K_2(\frac{ay}{b},y)$. If $(\sqrt{\frac{a^2}{b^2}+1})y-1$ or $(\sqrt{\frac{a^2}{b^2}+1})y+1$ divides ${((\frac{a^2}{b^2}+1)(y-b)^2 - r^2)}$, then $f_1=0$ intersects $f_2=0$ at $(\frac{a}{\sqrt{a^2+b^2}},\frac{b}{\sqrt{a^2+b^2}})$ or $(-\frac{a}{\sqrt{a^2+b^2}},-\frac{b}{\sqrt{a^2+b^2}})$. So, $(\frac{a^2}{b^2}+1)y^2-1)$ divides $K_2(\frac{ay}{b},y)$. Also, $y$ divides $K(\frac{ay}{b},y)$, otherwise $f_2=0$ intersects $x$-axis at $(0,0)$. Hence, $K_2(\frac{ay}{b},y)=0$ since $\deg K_2\leq 3$. Consequently, $bx-ay$ divides $K_2$. So, $K_2=\gamma xy(bx-ay)$ for some $\gamma\in \RR$. We rewrite \eqref{kolm-deg-4} and get the following. 
\begin{equation}\label{g2-inv-kolm-new}
    (x^2+y^2-1)(\alpha(x-a) + \beta(y-b)) + A(bx-ay)=\gamma(bx-ay)((x-a)^2 + (y-b)^2 - r^2).
\end{equation}
Hence, $bx-ay$ must divide $\alpha(x-a) + \beta(y-b)$. Then, $\alpha=\delta b$ and $\beta=-\delta a$ for some $\delta\in \RR$. So, rewriting \eqref{g2-inv-kolm-new}, we get
$$(x^2+y^2-1)\delta + A=\gamma((x-a)^2 + (y-b)^2 - r^2).$$
Since $\deg A\leq 1$, comparing the highest degree terms in the above equation, we obtain $\gamma=\delta$. So,
\begin{equation*}
P = xy(\gamma b(x^2+y^2-1) + Ay)~\mbox{and}~Q= xy(-\gamma a(x^2+y^2-1) - Ax).
\end{equation*}
Observe that the cofactor of $f_1=0$ is equal to $2K_2$. By Darboux Integrability Theory, $f_2 f_1^{-1}$ is a first integral of $\mathcal{X}$, which restricts $f_2=0$ to be a limit cycle. Hence, we get a contradiction. Thus, the proof is complete.
\end{proof} 

\noindent {\bf Acknowledgments.}
The authors thank J. Llibre for helpful discussions. The first author is supported by the Postdoctoral Fellowhip of IIT Bombay, India. The second author thanks `ANRF-India' for a MATRICS Research Grant ( ANRF/ARGM/2025/002640/MTR).


\end{document}